\documentclass[12pt]{amsart}

\usepackage{amsmath}
\usepackage{amssymb}
\usepackage{amstext}
\usepackage{amscd}
\usepackage[frenchb,english]{babel}
\usepackage{enumitem}
\usepackage{greekctr}
\usepackage[matrix,arrow,ps]{xy}
\usepackage[bookmarks=true,backref,colorlinks=true,citecolor=blue,urlcolor=blue,linkcolor=magenta]{hyperref}

\newtheorem{teo}{Theorem}[section]
\newtheorem{theorem}{Theorem}[section]

\newtheorem{proposition}[teo]{Proposition}

\newtheorem{lemma}[teo]{Lemma}

\newtheorem{corollary}[teo]{Corollary}

\newtheorem{conjecture}[teo]{Conjecture}

\newtheorem{defi}[teo]{Definition}
\newtheorem{definition}{Definition}

\newcommand{\Hom}{\mbox{Hom}}

\newcommand{\End}{{\rm End}}

\newcommand{\Sh}{{\rm Sh}}

\newcommand{\RR}{{\mathbb R}}
\newcommand{\R}{{\mathbb R}}

\newcommand{\Z}{{\mathbb Z}}
\newcommand{\QQ}{{\mathbb Q}}
\newcommand{\Q}{{\mathbb Q}}

\newcommand{\AAA}{{\mathbb A}}

\newcommand{\codim}{{\operatorname{codim}}}

\newcommand{\cH}{{\mathcal{H}}}

\newcommand{\ol}{\overline}

\usepackage{mathtools}

\DeclarePairedDelimiterX{\Nm}[1]{\lVert}{\rVert}{#1}
\newcommand{\sous}{\backslash}

\title[Generalisation of André-Oort and André-Pink-Zannier]{A common generalisation of the André-Oort and André-Pink-Zannier conjectures}
\author{Rodolphe Richard, Andrei Yafaev}
\date{\today}
 \address{Richard, Yafaev: UCL, Department of Mathematics, Gower street, WC1E 6BT, London, UK}
\email{Rodolphe.Richard@normalesup.org, yafaev@math.ucl.ac.uk}

\begin{document}
\setcounter{tocdepth}{1}
\setcounter{secnumdepth}{4}
\maketitle

\begin{abstract}
{We introduce a ``hybrid'' conjecture which is a common generalisation of the André-Oort conjecture and the André-Pink-Zannier conjecture and we prove that it is a consequence of the Zilber-Pink conjecture. We also show that our hybrid conjecture implies the Zilber-Pink conjecture for hypersurfaces contained in weakly special subvarieties.}
\end{abstract}

\tableofcontents

\section{Introduction}

In Diophantine geometry, Zilber-Pink type of questions have been a very active subject of research since mid 00s. In some aspects they are far reaching generalisations
of the Mordell-Lang conjecture for abelian varieties. Other motivations for these questions include 
applications in transcendence theory and model theory.

The most general recent results are the proof of the André-Oort conjecture in full generality (see \cite{AO-PST} and references therein), and André-Pink Zannier conjecture 
for Shimura varieties of abelian type (see \cite{APZ2} \cite{APZ2-CRAS}). 
These proofs rely on the so-called Pila-Zannier strategy which uses, among others, ideas and results from the theory of o-minimality and point counting, results on functional transcendence
(which are also proved using o-minimality).
We refer to the very recent monograph \cite{P} for an excellent overview of the methods and recent results on the Zilber-Pink type of questions.

In this paper we limit ourselves to the case of
\emph{ pure } Shimura varieties.  The Zilber-Pink conjecture can be stated for mixed Shimura varieties, or
even more generally for variations of $\mathbb{Z}$-Hodge structures.

\addtocontents{toc}{\protect\setcounter{tocdepth}{0}}
\section*{Contents of the article}
\addtocontents{toc}{\protect\setcounter{tocdepth}{1}}
We introduce a ``hybrid'' conjecture (Conjecture~\ref{conj:AO+APZ}) which is a common generalisation of the André-Oort conjecture and the André-Pink-Zannier conjecture, and we prove that it is a consequence of the Zilber-Pink conjecture. We ultimately show (after Lem.~\ref{lem:7.5}) that our hybrid conjecture implies the Zilber-Pink conjecture (as stated in Conjecture~\ref{conj:roro-ZP}) for hypersurfaces of weakly special subvarieties.


 In \S2 and \S3 we recall relevant facts about Shimura varieties, special and weakly special subvarieties as well as various notions of Hecke orbit. 
In \S4 we state the generalised Andr\'e-Pink-Zannier conjecture, the Andr\'e-Oort conjecture (we give several equivalent formulations) and the Zilber-Pink conjecture.

 
 Section 5 introduces the central object in Conjecture~\ref{conj:AO+APZ}: the notion, for a Shimura datum~$(M,X_M)$ and a Hodge generic point~$x\in X_M$,  of the \emph{hybrid Hecke orbit}~$\Sigma_{\Gamma\sous X^+}(M,X_M,x)$ of~$x$ in a Shimura variety~$\Gamma\sous X^+$. This hybrid Hecke orbit contains both the generalised Hecke orbit of~$x$ in~$\Gamma\sous X^+$ and special points of~$\Gamma\sous X^+$. 

We state Conjecture~\ref{conj:AO+APZ} and two equivalent forms of it (Conjectures~\ref{conj:AO+APZ-bis}  and~\ref{conj:AO+APZ-ter}). We then deduce it from the Zilber-Pink conjecture. These reformulations involve the set~$WS_{\Gamma\sous X^+}(M,X_M,x)$ of weakly special subvarieties of~$\Gamma\sous X^+$ which non-trivially intersect~$\Sigma_{\Gamma\sous X^+}(M,X_M,x)$. 

The sets~$\Sigma_{\Gamma\sous X^+}(M,X_M,x)$ and~$WS_{\Gamma\sous X^+}(M,X_M,x)$ have nice functoriality properties with respect to morphisms of Shimura data (see for instance \emph{Remarks} following Definition~\ref{def:orbite hybride},  Corollaries~\ref{cor:functorial1} and  Corollary~\ref{Functoriality}).

In \S\ref{sec:intersections} we study intersections of weakly special subvarieties with special subvarieties. We believe that our results here are of independent interest.

In~\S\ref{sec:relation to ZP} we relate Conjecture~\ref{conj:AO+APZ} to the Zilber-Pink conjecture.
We consider the product variety~$S:=\Gamma_M\sous X^+_M\times \Gamma\sous X^+$ and the weakly special subvariety~$W=\{\Gamma_M\cdot x\}\times \Gamma\sous X^+$. We prove that a subset~$V'\subseteq \Gamma\sous X^+$ is in~$WS_{\Gamma\sous X^+}(M,X_M,x)$ if and only if~$V:=\{\Gamma_M\cdot x\}\times V'$ is a component of the intersection of~$W$ and a special subvariety of~$S$.

In \S\ref{sec:AD}
we relate our hybrid Hecke orbits and hybrid conjecture to the recent work~\cite[\S3]{AD} of V. Aslanyan and C. Daw.
Their Conjecture~\cite[Conj.~3.4]{AD} is an immediate consequence of Conjecture~\ref{conj:AO+APZ}, but  Conjecture~\ref{conj:AO+APZ} is not a direct consequence of~\cite[Conj.~3.4]{AD}.

In~\S\ref{sec:ML}, we make explicit an analogy with Mordell-Lang conjecture in abelian varieties. This suggests an extension of our setting to mixed Shimura varieties.

\addtocontents{toc}{\protect\setcounter{tocdepth}{0}}
\section*{Acknowledgements}
\addtocontents{toc}{\protect\setcounter{tocdepth}{1}}
 The first author is very grateful to the IHES for hospitality and the second author is very grateful to the organisers of the BIRS Workshop `Moduli, Motives and Bundles' in Oaxaca, Mexico in September 2022.

\section{Shimura varieties and Weakly special subvarieties}

General references for Shimura varieties are \cite{Deligne} and~\cite{Milne}. Let $(G,X)$ be a Shimura datum in the sense of \cite{Deligne} and let $X^+$ be a connected component of $X$.
For every maximal compact subgroup~$K^+_{\infty}$ of~$G^{ad}(\RR)^+$, there exists a unique~$G^{ad}(\R)^+$ equivariant bijection~$X^+\simeq G^{ad}(\RR)^+/K^+_{\infty}$. 

Let~$\Gamma\subset G(\QQ)_+ = Stab_{G(\QQ)}(X^+)$ be a torsion-free arithmetic lattice. Then $\Gamma\backslash X^+$ is locally hermitian symmetric and has a canonical structure of quasi-projective algebraic variety
by a result of~Bailly-Borel~\cite{BailyBorel}. If~$\Gamma$ is a ``congruence lattice'', then~$\Gamma\sous X^+$ is a connected component of the Shimura variety $Sh_K(G,X)$ in the sense of \cite{Deligne} and it admits a canonical model over a certain 
explicit number field. By a common abuse of terminology we also simply refer to $\Gamma\backslash X^+$ as a Shimura variety. It is also refered to as \emph{arithmetic variety }.

\begin{definition}
A \emph{weakly special} subvariety of~$\Gamma\backslash X^+$ is a totally geodesic irreducible algebraic subvariety of~$\Gamma\backslash X^+$.
\end{definition}
With our definition every $0$-dimensional subvariety is a weakly special subvariety.

Following Moonen~\cite{Moonen} and Ullmo~\cite[Prop.~3.1]{U1},  every weakly special subvariety of~$\Gamma\backslash X^+$ can be written as
\begin{equation}\label{Eq-fs}
Z=\Gamma\backslash\Gamma\cdot H(\RR)^+\cdot x\subseteq \Gamma\backslash X^+
\end{equation}
for some $\QQ$-reductive subgroup~$H$ of $G$ and some point~$x\in X^+$ such that~$X^+_H:=H(\RR)^+\cdot x$ is a Hermitian symmetric subspace of~$X^+$.
 In this case,~$\Gamma_H=\Gamma \cap Stab_{H(\QQ)}(X_H^+)$ is an arithmetic lattice and, up to the finite morphism \(\Gamma_H\sous X^+_H\to Z\), the variety~$Z$ is an arithmetic variety.

\begin{definition}\label{def:MT}
Let~$x_0\in X$ and let~$(M,M(\R)\cdot x_0)$ be the smallest Shimura subdatum~$(H,X_H)\leq (G,X)$ such that~$x_0\in X_H$. We refer to~$M$ as the \emph{Mumford-Tate group} of~$x_0$, to~$(M,M(\R)\cdot x_0)$ as the \emph{Mumford-Tate datum} of~$x_0$, and call~$(M,M(\R)\cdot x_0,x_0)$ the \emph{pointed Mumford-Tate datum} of~$x_0$.
\end{definition}

\begin{definition}\label{def:special}
A point~$x\in X$ is \emph{special} if its Mumford-Tate group is an algebraic torus.
 A \emph{special point} of~$\Gamma\sous X^+$ is the image of a special~$x\in X^+$.
A \emph{special} subvariety~$Z\subseteq \Gamma\sous X^+$ is a weakly special subvariety containing a special point.
\end{definition}


%
%

\section{Various notions of Hecke orbits}
We consider a Shimura datum~$(G,X)$ and a point~$x_0\in X$.
\begin{defi}\label{defi:Hecke}
We define the \emph{classical Hecke orbit} of~$x_0$ to be the following subset of $X$:
\[
\cH^c(x_0):=G(\QQ)\cdot x_0.
\]
After identifying~$G$ with its image under a faithful representation~$\rho:G\to GL(n)$, we define the~\emph{$\rho$-Hecke orbit} of~$x_0$
as
\[
\cH^{\rho}
(\rho(GL(n,\QQ))\cdot x_0)\cap X.
\]
\end{defi}
The image of~$\cH^c(x_0)$ in~$\Gamma\sous X^+$ is the orbit of~$x_0$ under Hecke correspondences~$\Gamma\sous X^+\xleftarrow{}(\Gamma\cap \Gamma')\sous X^+\xrightarrow{} \Gamma'\sous X^+\simeq \Gamma\sous X^+ $ with~$\Gamma'=g\Gamma g^{-1}$ and~$g\in G(\Q)$. One issue with the notion of classical Hecke orbit is that it does not behave well under passing to the adjoint group: if~$G$ is semi-simple with non-trivial centre,  then~$(G^{ad}(\QQ)\cdot x_0)\cap X$ is the union infinitely many disjoint classical Hecke orbits. 

In  moduli spaces of polarised abelian varieties, the classical Hecke orbit of a moduli point~$[A]$ of a polarised abelian variety~$A$ is the set of~$[B]$ where~$B$ is such that there exists an isogeny~$A\to B$ \emph{which is compatible with the polarisations}. In order to allow all isogenies, we use the~$\rho$-Hecke orbit for the natural representation of~$G=GSp(2g)\leq GL(2g)$. We refer to~\cite{OrrThesis} for more details.

The following is defined in~\cite{APZ1}. 
We let~$\Hom(M,G)(\QQ)$ be the set of algebraic group morphisms \emph{defined over~$\QQ$}. 
\begin{definition}[{\cite[]{APZ1}}] \label{gho}
The \emph{Generalised Hecke orbit of $x_0$ in $X$} is
$$
\cH_X(x_0) := X\cap \{ \phi \circ  x_0 :  \phi \in  \Hom(M,G)(\QQ)\}
$$
and the \emph{Generalised Hecke orbit of $x_0$ in~$\Gamma\sous X^+$} is
$$
\cH_{\Gamma\sous X^+}(x_0) := \Gamma\sous (\cH_X(x_0)\cap X^+).
$$
\end{definition}
Another useful notion of Hecke orbit is the following. We denote by~$\phi_0:M\to G$ the natural inclusion homomorphism
and by~$(G\cdot \phi_0)(\QQ)=(G(\ol{\QQ})\cdot \phi_0)\cap\Hom(M,G)(\QQ)$ the set of morphisms~$\phi:M\to G$ which are defined over~$\Q$ and conjugated to~$\phi_0$ by an element of~$G(\ol{\QQ})$.
\begin{defi} \label{geoho}
We define the \emph{Geometric Hecke orbit} $\cH(x_0)$ of $x_0$ \emph{in $X$} as
$$
\cH^g(x_0) := X\cap \{ \phi \circ  x_0 :  \phi \in  (G\cdot \phi_0)(\QQ)\}
$$
and the \emph{Geometric Hecke orbit} $\cH^g([x_0, g_0])$ of $[x_0, g_0]$ \emph{in $\Sh_K(G,X)$} as
$$
\cH^g([x_0,g_0]) := \{ [x,g] : x\in \cH^g(x_0), g\in G(\AAA_f) \}.
$$
\end{defi}
We have the following inclusions 
\[
\cH^c(x_0)\subseteq\cH^g(x_0)\subseteq \cH^{\rho}(x_0)\subseteq \cH(x_0).
\]
Only the second inclusion is not obvious. It follows from~\cite[Th.\, 2.11]{APZ1}.

The reason for introducing the generalised and the geometric Hecke orbits is that, unlike the classical Hecke orbit, they have natural functoriality properties with respect to
morphims of Shimura data (see \cite{APZ1}). 

A fundamental property (see~\cite[Th.~2.4]{APZ1}) is the following
\begin{theorem}[{see~\cite[Th.~2.4]{APZ1}}]\label{thm:finiteness geometric}
Any generalised Hecke orbit is a finite union 
of geometric Hecke orbits.
\end{theorem}



\section{André-Pink-Zannier, André-Oort and Zilber-Pink conjectures}
The following conjecture is stated in~\cite{APZ1}.
\begin{conjecture}[Generalised André-Pink-Zannier conjecture]\label{conj:APZ}
 Let~$V$ be an algebraic subvariety of a Shimura variety~$\Gamma\sous X^+$. Let~$x\in X^+$ and let~$\cH_{\Gamma\sous X^+}(x)$ be its generalised Hecke orbit.
 
Then the Zariski closure of~$V\cap \cH_{\Gamma\sous X^+}(x)$ is a finite union of weakly special subvarieties of~$\Gamma\sous X^+$.
\end{conjecture}

A major recent achievement is a proof of the following conjecture 
(see~\cite{AO-PST}, historical remarks  and references contained therein).
\begin{conjecture}[André-Oort conjecture]\label{conj:AO} Let~$V$ be an algebraic subvariety of an arithmetic variety~$\Gamma\sous X^+$, and let~$\Sigma$ be the set of all special points of~$\Gamma\sous X^+$ belonging to~$V$.

Then the Zariski closure of~$\Sigma$ is a finite union of special subvarieties of~$\Gamma\sous X^+$.
\end{conjecture}

As every special subvariety contains a Zariski dense set of special points, 
Conjecture~\ref{conj:AO} is equivalent to the following.
	\begin{conjecture}[Reformulation of André-Oort conjecture]\label{conj:AO-bis} Let~$V$ be an algebraic subvariety of a Shimura variety~$\Gamma\sous X^+$, and let~$\{Z_1;\ldots\}$ be the set of all special subvarieties~$Z_i\subseteq \Gamma\sous X^+$ such that~$Z_i\subseteq V$.

Then the Zariski closure of~$\bigcup Z_i$ is a finite union of special subvarieties of~$\Gamma\sous X^+$. In other words, the set~$\{Z_1;\ldots\}$ has finitely many maximal elements with respect to inclusion.
\end{conjecture}
Recall that irreducible components of the intersection of two special subvarieties are special varieties. 
Thus, in Conjectures~\ref{conj:AO} and~\ref{conj:AO-bis} we may assume
\[
\text{ $\Gamma\sous X^+$ is the smallest special subvariety containing~$V$. }
\]
Conjecture~\ref{conj:AO-bis} for a given~$V$ is a consequence of Conjecture~\ref{conj:AO-ter} applied to the irreducible components of~$\ol{\bigcup Z_i}^{Zar}$. Conjecture~\ref{conj:AO-ter} for a given~$V$ is a consequence of Conjecture~\ref{conj:AO-bis}.
\begin{conjecture}[Another formulation of André-Oort conjecture]\label{conj:AO-ter}
Let~$V$ be an algebraic subvariety of an arithmetic variety~$\Gamma\sous X^+$.

Let~$\{Z_1;\ldots\}$ be the set of all special subvarieties~$Z_i\subseteq \Gamma\sous X^+$ 
such that~$Z_i\subseteq V$.

If~$\bigcup Z_i$ is Zariski dense in~$V$, then~$V$ is a special subvariety.
\end{conjecture}

We will formulate in~\S\ref{sec:hybrid} a ``hybrid conjecture'' (Conjecture~\ref{conj:AO+APZ},\ref{conj:AO+APZ-bis},\ref{conj:AO+APZ-ter}) which is a common generalisation of both the André-Oort conjecture and the generalised André-Pink-Zannier conjecture. 

We will see in~\S\ref{sec:relation to ZP} that the hybrid conjecture is a consequences of Zilber-Pink conjectures on atypical intersections in Shimura varieties, which we now formulate.
\begin{definition}\label{def:atypical}
Let~$S=\Gamma\sous X^+$ be an arithmetic variety and let~$Z,V\subseteq S$ be two algebraic subvarieties.
An irreducible component~$C$ of~$Z\cap V$ is said to be an \emph{atypical intersection of~$Z$ and~$V$ with respect to~$S$} if
\[
\codim_S(C)< \codim_S(Z)+\codim_S(V)
\]
or, equivalently,
\[
\dim(C)> \dim(Z)+\dim(V)-\dim(S).
\]
We say that~$C$ is a \emph{typical intersection of~$Z$ and~$V$ with respect to~$S$} if~$\codim_S(C)= \codim_S(Z)+\codim_S(V)$.
\end{definition}
We refer to~\cite{P} for more details on the following conjecture.
\begin{conjecture}[Zilber-Pink conjecture]\label{conj:roro-ZP}
Let~$V$ be a proper algebraic subvariety of an arithmetic variety~$\Gamma\sous X^+$, and let~$(Z_n)_{n\in\Z_{\geq0}}$ be a sequence of special subvarieties of~$S$ such that, for each~$n\in \Z_{\geq0}$ there exists a component~$W_n$ of~$V\cap Z_n$ which is an atypical intersection of~$Z_n$ and~$V$ in~$\Gamma\sous X^+$.

Then~$\bigcup W_n$ is contained in a finite union of proper special subvarieties (i.e is not Hodge generic).
\end{conjecture}

\subsubsection*{Remarks} We make the following remarks.
\begin{enumerate}
\item If~$V$ is contained in a proper special subvariety of~$S$, then the conclusion is obvious.
\item Conjecture~\ref{conj:roro-ZP} implies that: If~$V$ is not contained in a proper special subvariety of~$S$, then
\[
\text{$\bigcup W_n$ is not Zariski dense in~$V$.}
\]
\item After extracting a subsequence we can assume that~$\dim(Z_n)$ and~$\dim(W_n)$ are independent of~$n$.
We may also assume that~$V$ is not contained in a proper special subvariety of~$S$.
\item In the case~$\dim(Z_n)=\dim(W_n)$ we have an atypical intersection if and only if~$Z_n\subseteq V$ and~$V\neq S$.
Thus, the case~$\dim(Z_n)=\dim(W_n)$ of Conjecture~\ref{conj:AO-ter} for~$V$ is equivalent to Conjecture~\ref{conj:roro-ZP} for~$V$.
\item When~$V$ is a hypersurface of~$\Gamma\sous X^+$, every atypical intersection satisfies~$\codim_{Z_n}(W_n)<\codim_S(V)=1$. Thus the André-Oort conjecture implies the case~$\dim(V)=\dim(\Gamma\sous X^+)-1$ of~ Conjecture~\ref{conj:roro-ZP}.
\end{enumerate}

\section{Hybrid Hecke orbits and the Hybrid conjecture}\label{sec:hybrid}

For a Shimura datum~$(G,X)$ and a point~$x'\in X$, we denote by~$(M_{x'},X_{x'})$ the Mumford-Tate datum of~$x'$ as in Definition~\ref{def:MT}. We denote by~$Z(M)$ the centre of an algebraic group~$M$ and~$M^{ad}:=M/Z(M)$ the associated adjoint group. 

In the following definition, there is no need to assume that~$(M,X_M)$ is a subdatum of~$(G,X)$.
\begin{definition}\label{def:orbite hybride}
 Let~$(G,X)$ and~$(M,X_M)$ be Shimura data and let~$x\in X_M$ be a Hodge generic point. The \emph{hybrid Hecke orbit of~$(M,X_M,x)$ in~$X$} is
\begin{multline}\label{definition sigma hybride}
\Sigma_X(M,X_M,x)=\{x'\in X|\exists~\phi\in \Hom(M^{ad},M^{ad}_{x'}),
\phi(x^{ad})=x'^{ad}\}.
\end{multline}
For an arithmetic variety~$\Gamma\sous X^+$, we write
\[
\Sigma_{\Gamma\sous X^+}(M,X_M,x):=\Gamma\sous (\Sigma_X(M,X_M,x)\cap X^+).
\]
\end{definition}
\subsubsection*{Remarks}\label{rem:hybrid}
We note the following important properties:
\begin{enumerate}
\item \label{rem:hybrid0} 
Every~$\phi\in \Hom(M^{ad},M^{ad}_{x'})$ such that~$\phi(x^{ad})=x'^{ad}$ is an epimorphism. This follows from the fact that
the Mumford-Tate group of~$\phi(x^{ad})=x'^{ad}$ is~$\phi(M^{ad})=M^{ad}_{x'}$.
\item \label{rem:hybrid1}
If~$x'\in X$ is a special point, we have~$M^{ad}_{x'}=\{1\}$, and we deduce~$x'\in \Sigma_X(M,X_M,x)$.
Indeed, we can take~$\phi$ to be the constant morphism~$\phi:M^{ad}\to \{1\}$.
\item \label{rem:hybrid2}
If~$x$ is a special point, we have~$M^{ad}=\{1\}$, and we have~$x'\in \Sigma_X(M,X_M,x)$ if and only if~$x'$ is a special point of~$X$. This is because~$M^{ad}_{x'}=\phi(M^{ad})=\{1\}$.
\item \label{rem:hybrid3}
\emph{Transitivity:} If~$x'\in \Sigma_X(M,X_M,x)$, then~$\Sigma_{X'}(M_{x'},X_{x'},x')\subseteq \Sigma_{X'}(M,X_M,x)$ for every Shimura datum~$(G',X')$.
\item \label{rem:hybrid4}
If~$x''\in \cH(x')$, then we have an epimorphism~$M_{x'}\to M_{x''}$ such that~$\phi(x')=x''$. We deduce that
\begin{equation}\label{eq:rem:hybrid4}
\forall x'\in \Sigma_X(M,X_M,x),\cH(x')\subseteq \Sigma_X(M_{x'},X_{x'},x').
\end{equation}
\end{enumerate}
It follows from~\eqref{eq:rem:hybrid4} and the Proposition~\ref{prop:countable} below
that~$\Sigma_X(M,X_M,x)$ is a countable union of generalised Hecke orbits which contains all special points of~$X$.

\begin{proposition}\label{prop:countable}
The set~$\Sigma_X(M,X_M,x)$ is countable.
\end{proposition}
\begin{proof}Consider the semisimple adjoint group~$M^{ad}=\prod_{i\in I}M_i$ decomposed into a product of~$\Q$-simple factors. Correspondingly, write~$X^{ad}=\prod_{i\in I}X_i$. For every epimorphism~$\phi:M^{ad}\to M'$, there exists~$J\subseteq I$ and an isomorphism~$\iota:M_J:=\prod_{i\in J} M_j\to M'$ such that~$\phi=\iota\circ p_J:M^{ad}\xrightarrow{p_J} M_J\xrightarrow{\iota} M'$, where~$p_J$ denotes the natural quotient map. 

This implies that for~$x'\in \Sigma_X(M,X_M,x)$, there exists a subset~$J\subseteq I$ and an isomorphism of Shimura data~$(M^{ad}_{x'},X^{ad}_{x'})\to (M_J,X_J,p_J(x^{ad}))$ mapping~$p_J(x^{ad})$ to~$x'^{ad}$, where~$X_J:=\prod_{i\in J} X_i$.

There are countably many Shimura subdata~$(H,X_H)\leq (G,X)$, and for each Shimura datum~$(H,X_H)$ and every~$J\subseteq I$, there are countably many morphisms of Shimura data~$\phi:(H,X_H)\to (M_J,X_J)$. When~$\phi$ is an epimorphism inducing~$H^{ad}\simeq M_J$, the map~$X_H\to X_J$ is an injection, and there is at most one~$x'\in X_H$ such that~$\phi(x')=p_J(x)$.

Note that every~$x'\in \Sigma_X(M,X_M,x)$ is of this form for~$(H,X_H)=(M_{x'},X_{x'})\leq (G,X)$, for some~$J\subseteq I$ and some~$\phi:(H,X_H)\to (M_J,X_J)$. Therefore, there are at most countably many~$x'\in \Sigma_X(M,X_M,x)$.

\end{proof}

The following is a common generalisation of André-Oort and André-Pink-Zannier conjectures.
\begin{conjecture}[Hybrid conjecture]\label{conj:AO+APZ}
Let~$V$ be an irreducible algebraic subvariety of an arithmetic variety~$\Gamma\sous X^+$,
let~$(M,X_M)$ be a Shimura datum and let~$x\in X_M$ be a Hodge generic point.

Then the Zariski closure of~$\Sigma_{\Gamma\sous X^+}(M,X_M,x)\cap V$ is a finite union of weakly special subvarieties of~$\Gamma\sous X^+$.
\end{conjecture}
If~$x$ is a special point, then the set~$\Sigma_{\Gamma\sous X^+}(M,X_M,x)$ of Conjecture~\ref{conj:AO+APZ} is the same as the set~$\Sigma$ of Conjecture~\ref{conj:AO}, and every weakly special subvariety in the conclusion of Conjecture~\ref{conj:AO+APZ} will be special by Definition~\ref{def:special}. Thus, if~$x$ a special point, then Conjecture~\ref{conj:AO+APZ} is equivalent to Conjecture~\ref{conj:AO}. Since the André-Oort conjecture has been proven, Conjecture~\ref{conj:AO+APZ} is true when~$x$ is a special point.

For a non necessarily special~$x\in X$, as we have~$\cH(x)\subseteq \Sigma_X(M,X_M,x)$, Conjecture~\ref{conj:APZ} is a consequence of Conjecture~\ref{conj:AO+APZ}.

Recall that for any~$x\in X$, the set~$\Sigma_X(M,X_M,x)$ contains all special points of~$X$. Thus, for any~$x\in X$,
Conjecture~\ref{conj:AO+APZ} for~$x\in X$ implies the André-Oort Conjecture~\ref{conj:AO}.

Let~$WS_X$ be the set of weakly special subvarieties (possibly of dimension~$0$) of~$X$ and let us define
\[
WS_X(M,X_M,x)=\{W\in WS_X| W\cap \Sigma_X(M,X_M,x)\neq \emptyset\}.
\]
Let
\[
WS_{\Gamma\sous X^+}(M,X_M,x)=\{\Gamma\sous \Gamma W|W\in WS_X(M,X_M,x), W\subseteq X^+\}.
\]
Thus~$WS_{\Gamma\sous X^+}(M,X_M,x)$ is the set of weakly special subvarieties~$W\subseteq \Gamma\sous X^+$
such that~$W\cap \Gamma\sous \Sigma_X(M,X_M,x)\neq \emptyset$.

\begin{conjecture}[Second version]\label{conj:AO+APZ-bis}
Let~$V$ be an irreducible algebraic subvariety of an arithmetic variety~$\Gamma\sous X^+$,
let~$(M,X_M)$ be a Shimura datum and let~$x\in X_M$ be a Hodge generic point.

Let~$\{Z_1;\ldots\}$ be the set of all subvarieties~$Z_i\in WS_{\Gamma\sous X^+}(M,X_M,x)$ such that~$Z_i\subseteq V$.

Then there exists~$W_1,\ldots,W_k\in WS_{\Gamma\sous X^{+}}(M,X_M,x)$ such that the Zariski closure of~$\bigcup Z_i$ is~$W_1\cup\ldots\cup W_k$. In other words,~$\{Z_1;\ldots\}$ contains finitely many maximal elements with respect to inclusion.
\end{conjecture}
Note that we have
\[
\Sigma_{\Gamma\sous X^+} (M, X_M , x) \cap  V\subseteq \bigcup \{Z\in WS_{\Gamma\sous X^+}(M,X_M,x)| Z\subseteq V\}
\]
and that, in view of \cite[Lem. 3.3]{AD} the left hand-side is dense in the right hand-side for the Archimedean topology.

It follows that Conjecture~\ref{conj:AO+APZ-bis} is equivalent to Conjecture~\ref{conj:AO+APZ}.

Conjecture~\ref{conj:AO+APZ-bis} for a given~$V$ is a consequence of Conjecture~\ref{conj:AO+APZ-ter} applied to irreducible components of~$\ol{\bigcup Z_i}^{Zar}$.
\begin{conjecture}[Third version]\label{conj:AO+APZ-ter}
Let~$V$ be an irreducible algebraic subvariety of an arithmetic variety~$\Gamma\sous X^+$,
let~$(M,X_M)$ be a Shimura datum and let~$x\in X_M$ be a Hodge generic point.


Assume that~$\bigcup \{Z\in WS_{\Gamma\sous X^+}(M,X_M,x)| Z\subseteq V\}$ is Zariski dense in~$W$. Then~$W$ is weakly special, and~$W$ belongs to~$WS_{\Gamma\sous X^+}(M,X_M,x)$.
\end{conjecture}

\section{Intersection of a weakly special subvariety with special varieties}\label{sec:intersections}

In this section we study intersections of special and weakly special subvarieties. We notably show that, under certain conditions, 
 such intersections are `typical'. We believe that  results below are of independent interest and use.

\begin{theorem}\label{thm:special intersections are typical}
 Let~$(M,X_M)$ and~$(G,X)$ be two Shimura data and let~$W=\{x\}\times X\subseteq X_M\times X$ with~$x\in X_M$ a Hodge generic point.

Then, for every Shimura subdatum~$(H,X_H)\leq (M\times G ,X_M\times X)$ such that~$X_H\cap W\neq \emptyset$,
we have
\[
\dim(X_H\cap W)
=\dim(X_H)+\dim(W)-\dim(X_M\times X)
.
\]
In the sense of Definition~\ref{def:atypical} the intersection of~$X_H$ and~$W$ is typical with respect to~$X_M\times X$.
\end{theorem}
\begin{proof}
As~$\dim(X_M\times X)=\dim(X_M)+\dim(X)=\dim(X_M)+\dim(W)$, it is enough to prove that
\[
\dim(X_M)+\dim(X_H\cap W)=\dim(X_H).
\]
Let~$\phi:H\to M$ and~$\pi:X_H\to X_M$ be the projections.
For~$h\in H(\R)$ and~$y\in X_H$ we have
\[
\phi(h)\cdot \pi(y)=\pi(h\cdot y).
\]
As~$H(\R)$ acts transitively on~$X_H$, it follows that~$H(\R)$ permutes transitively the fibres of~$\pi$.
Thus~$\pi:X_H\to \pi(X_H)$ is a fibration. Note that~$X_H\cap W$ is a fibre of~$\pi$. We deduce that
\[
\dim(X_H)=\dim(\pi(X_H))+\dim(X_H\cap W).
\]
As~$X_H\cap W\neq \emptyset$, we have~$x\in \pi(X_H)$. As~$x$ is Hodge generic in~$X_M$, the connected component~$X_M^+\subseteq X$ containing~$x$ satisfies~$X_M^+\subseteq \pi(X_H)$. We deduce that~$\dim(\pi(X_H))=\dim(X_M)$ and the conclusion follows.
\end{proof}

\begin{corollary}\label{cor:special intersections are typical}
Let~$S=\Gamma\sous X^+$ be an arithmetic variety, let~$W\subseteq S$ be a Hodge generic weakly special subvariety,
and let~$Z\subseteq S$ be a special subvariety. Then every irreducible component of~$Z\cap W$ is a typical intersection of~$Z$ and~$W$ with respect to~$S$.
\end{corollary}
\begin{proof} We identify~$X^+$ with a component of~$X^{ad}$. As~$W$ is a Hodge generic weakly special subvariety, there exists a factorisation
\[
(G^{ad},X^{ad})=(G_1,X_1)\times (G_2,X_2)\text{ and }X^+=X_1^+\times X_2^+
\]
and~$x_1\in X_1^+$ such that~$W$ is the image of~$W^+:=\{x_1\}\times X_2^+$ in~$\Gamma\sous X^+$.
As~$Z$ is a special subvariety, there exists a subdatum~$(H,X_H)\leq (G^{ad},X^{ad})$ and a component~$X^+_H\subseteq X^+\cap X_H$ such that~$Z$ is the image of~$X_H^+$.

Let~$D$ be an irreducible component of~$Z\cap W$ and let~$s\in D$ be a point not lying on another irreducible component and such that~$D$ is smooth at~$s$.
Then there exists a connected open neighbourhood~$U$ of~$s$ such that~$U\cap Z\cap W=U\cap D$ and~$U\cap D$ is an irreducible closed complex analytic subvariety of~$U$. 

There exists~$w\in W^+$ such that~$s$ is the image of~$w$ in~$\Gamma\sous X^+$. As~$X^+\to \Gamma\sous X^+$ is an étale cover (recall that $\Gamma$ is torsion free), there exists a connected neighbourhood~$V$ of~$x$ such that the image of~$V$ is~$U$ and that the map~$V\to U$ is a complex analytic isomorphism.

The inverse image of~$Z\cap W$ in~$V$ is~$\Gamma\cdot X_H^+\cap \Gamma\cdot W^+\cap V$. Irreducible components of~$X^+_H\cap W^+$ are closed complex analytic subvarieties of~$X^+$, and it follows from Theorem~\ref{thm:special intersections are typical} that these varieties are all of dimension~$d:=\dim(X_H)-\dim(X_1)$. We deduce that~$\Gamma\cdot X_H^+\cap \Gamma\cdot W^+\cap V$ is a countable union of irreducible closed complex analytic subvarieties of dimension~$d$. Applying the isomorphism~$V\to U$, we deduce that the irreducible complex analytic variety~$D\cap U$ is a countable union of irreducible closed complex analytic subvarieties of dimension~$d$. This implies~$\dim(D\cap U)=d$. Thus~$\dim(D)=d$.

We observe that~$\dim(Z)=\dim(X_H)$ and~$\dim(W^+)=\dim(X_1)$. We deduce that the component~$D$ is a typical intersection of~$Z\cap W$ with respect to~$\Gamma\sous X^+$. As~$D$ was chosen arbitrarily, this is also true for every component~$D$ of~$Z\cap W$.\end{proof}

\section{Relation between the hybrid conjecture and the Zilber-Pink conjecture}\label{sec:relation to ZP}
\begin{theorem}\label{thm:weakly inter special}Let~$(M,X_M)$ and~$(G,X)$ be Shimura data, let~$x\in X_M$
be a Hodge generic point and let~$W:=\{x\}\times X\subseteq X_M\times X$.
For~$V\subseteq W$ the following are equivalent.
\begin{enumerate}
\item \label{thm:weakly inter special 1}
There exists a Shimura subdatum~$(H,X_H)\leq (M\times G, X_M\times X)$  such that~$V$ is a component of~$X_H\cap W$.
\item \label{thm:weakly inter special 2} There exists~$V'\in WS_X(M,X_M,x)$ such that~$V=\{x\}\times V'$.
\end{enumerate}
\end{theorem}
\begin{proof}[Proof of  the implication~\eqref{thm:weakly inter special 1}$\Rightarrow$\eqref{thm:weakly inter special 2}]
Let~$(H,X_H)\leq(M\times G, X_M\times X)$ be a subdatum such that~$X_H\cap W\neq \emptyset$, and let~$V$ be a component of~$X_H\cap W$. As~$V\subseteq W$, we can write~$V=\{x\}\times V'$ for some~$V'\subseteq X$.

Since the components~$X_H$ and~$W$ are totally geodesic and complex analytic,~$V$ is also totally geodesic and complex analytic.
It follows that~$V\in WS_X$.
We need to prove that
\begin{equation}\label{proof 1 implies 2}
\exists x'\in \Sigma_X(M,X_M,x), (x,x')\in V.
\end{equation}
To do this, we will construct a subdatum~$(L',L'(\R)\cdot s)\leq (H,X_H)$ such that the morphism~$L'\to M$ is surjective and induces an isomorphism~$L'^{ad}\to M^{ad}$. We will show that~$X_L\cap V$ is of the form~$X_L\cap V=\{(x,x')\}$ and that~$x'\in \Sigma_X(M,X_M,x)$.


Let~$Z$ be the connected component of~$X_H$ containing~$V$, and consider~$s$ in~$X_H$, a special point within~$Z$. Recall that~$T:=M_{s}\leq H$, being a torus,  forms a Shimura subdatum~$(T,\{s\})\leq (H,X_H)$. This implies that the torus~$T\leq H$ normalises the normal subgroups~$L:=G\cap H=\ker(H\to M)$ and~$Z_H(L)$ of~$H$. It follows that~$L':=Z_H(L)^{0}\cdot T\leq H$ is a reductive subgroup.

As~$X_H\cap W\neq \emptyset$, the image of~$X_H\to X_M$ contains the Hodge generic point~$x$. Consequently, the image of~$X_H\to X_M$ includes the connected component of~$X_M$ that contains~$x$, and~$H\to M^{ad}$ is surjective.

Given that~$H$ is reductive and~$L=\ker(H\to M)\leq H$ is a connected \emph{normal} subgroup, we have~$H=L\cdot Z_H(L)^0$. This implies that we can express~$L'=Z_H(L)^0\cdot S$ where~$S\leq L$ is a torus. 

Furthermore, this implies that the morphism~$Z_H(L)^0\to M$ is surjective. Consequently, the morphism~$L'\to M$ is surjective. The kernel of~$Z_H(L)\to M$ is~$L\cap Z_H(L)=Z(L)\leq Z(Z_H(L))$ and the kernel of~$L'\to M$ is~$S\cdot Z(L)\subseteq Z(Z_H(L))\cap Z(S)=Z(L')$. We deduce that the epimorphism~$L'\to M$ is a central extension. As a result, the induced map~$L'^{ad}\to M^{ad}$ is an isomorphism. 

Note that no~$\Q$-factor of~$M^{ad}$ is~$\R$-anisotropic. This, in turn, implies that no~$\Q$-factor of~$L'^{ad}$ is~$\R$-anisotropic. According to~\cite[Lemme 3.3]{U} this implies that~$(L',L'(\R)\cdot s)\leq (H,X_H)$ is a Shimura subdatum.

As the induced map~$L'^{ad}\to M^{ad}$ is an isomorphism, the map~$L'(\R)^+\cdot s\to X_M$ is injective, and its image is a component of~$X_M$. Since~$s\in Z$, this image contains the component of~$X_M$ containing the image of~$Z$, which is also the component of~$X_M$ containing~$x$. Therefore, there exists an element~$(x,x')\in W\cap L'(\R)^+\cdot s$. Given that~$L'(\R)^+\cdot s\subseteq Z$, we have~$(x,x')\in W\cap Z=V$.

The morphism~$M\times G\to M$ induces a surjection of pointed Mumford-Tate data~$(M^{ad}_{(x,x')},X^{ad}_{(x,x')},(x,x')^{ad})\to (M^{ad}_{x},X^{ad}_{x},x^{ad})$. Since~$(x,x')\in L(\R)\cdot s$, we have~$M_{(x,x')}\leq L$. Consequently, the surjection~$M^{ad}_{(x,x')}\to M^{ad}_{x}$ is an isomorphism. By definition,
\[
(x,x')\in \Sigma_{X_M\times X}(M,X_M,x).
\]

The morphism of Shimura data~$(L,L(\R)\cdot s)\to (G,X)$ maps~$(x,x')$ to~$x'$. Similarly, the morphism of Shimura data~$(M_{(x,x')},X_{(x,x')})\to (G,X)$ maps~$(x,x')$ to~$x'$. By definition~$x'$ is in the generalised Hecke orbit of~$(x,x')$.
According to remarks~\ref{rem:hybrid4} and~\ref{rem:hybrid3} following Definition~\ref{def:orbite hybride}, we have
\[
x'\in \Sigma_{X}(M_{(x,x')},X_{(x,x')},(x,x'))\subseteq \Sigma_{X}(M,X_M,x).
\]
This implies~\eqref{proof 1 implies 2} and concludes the proof of the implication~\eqref{thm:weakly inter special 1}$\Rightarrow$\eqref{thm:weakly inter special 2}.
%
%
\end{proof}
\begin{proof}[We now prove the implication~\eqref{thm:weakly inter special 2}$\Rightarrow$\eqref{thm:weakly inter special 1}]

Consider a weakly special~$V'\subseteq X$. There exists a Shimura subdatum~$(L,X_L)\leq (G,X)$ and a factorisation~$(L^{ad},X^{ad}_L)=(L_1\times L_2,X_1\times X_2)$ and~$x_1\in X_1$ such that the image of~$V$ by~$X_L\hookrightarrow X_L^{ad}$ is a component of~$\{x_1\}\times X_2$. Then~$V'$ is an orbit of~$L_2(\R)^+$.

Without loss of generality, let us assume~$M=M^{ad}$ and~$G=L=L^{ad}=L_1\times L_2$, and that~$x_1$ is Hodge generic in~$X_1$.

By assumption, there exists~$x'=(x_1,x_2)\in V'\cap \Sigma_X(M,X_M,x)$. Since~$x'\in \Sigma_X(M,X_M,x)$, there exists a morphism of Shimura data~$\phi':(M^{ad},X_M^{ad})\to (M_{x'}^{ad},X^{ad}_{x'})$ such that~$\phi'\circ x^{ad}=x'^{ad}$. The morphism~$M_{x'}\to L_1$ is surjective because~$x_1$ is the image of~$x'$ and is Hodge generic in~$X_1$. As~$L_1=L_1^{ad}$ this induces a morphism~$\phi'':M^{ad}_{x'}\to L_1$.

Then~$\phi=\phi''\circ\phi':(M^{ad},X_M^{ad})\to (M_{x'}^{ad},X^{ad}_{x'})\to (L_1,x_1)$ is such that~$\phi\circ x=x_1$.

Then the morphism~$\Phi:M\times L_2\to M\times L_1\times L_2$ defined by
\[
\Phi(m,h)=(m,\phi(m),h)
\]
gives us a morphism of Shimura data
\[
(M\times L_2,X_M\times X_2)\to (M\times L_1\times L_2,X_M\times X_1\times X_2).
\]
Let~$(H,X_H)$ be the Shimura datum which is the image of this morphism, where~$H=\Phi(M\times L_2)$ and~$X_H=H(\R)\cdot \Phi(X_M\times X_2)$.
Note that~$\Phi(X_M\times X_2)$ is a union of components of~$X_H$.

An element~$(a,b,c)\in X_M\times X_1\times X_2$ satisfies
\[
(a,b,c)\in \left(\{x\}\times X_1\times X_2\right)\cap \Phi(X_M\times X_2)
\]
if and only if~$a=x$ and~$b=\phi(a)=\phi(x)=x_1$. We thus have
\[
\left(\{x\}\times X_1\times X_2\right)\cap \Phi(X_M\times X_2)=\{x\}\times\{x_1\}\times X_2.
\]
Components of this intersection are orbits of~$L_2(\R)^+$. The orbit containing~$(x,x_1,x_2)\in V=\{x\}\times V'$ is thus equal to~$V$. We can deduce that~$V$ is a component of~$X_H\cap W$. This finishes the proof.
\end{proof}

We obtain a new interpretation of the~$Z\in WS_X(M,X_M,x)$, as images of the~$x^{ad}$ by general correspondences induced by morphisms of Shimura data.
\begin{corollary}\label{cor:functorial1}
For~$V'\subseteq X$, we have~$V'\in WS_{X}(M,X_M,x)$ if and only if there exists a Shimura datum~$(H,X_H)$ and two morphisms of Shimura data~$\Phi_1:(H,X_H)\to (M^{ad},X^{ad}_M)$ and~$\Phi_2:(H,X_H)\to (G,X)$ such that
\[
V'\text{ is a component  of }\Phi_2\left(\stackrel{-1}{\Phi_1}(\{x^{ad}\})\right).
\]
\end{corollary}
\begin{proof}Let us prove the first implication.
Suppose that~$V'\in WS_{X}(M,X_M,x)$ and let~$V:=\{x^{ad}\}\times V'$ and~$W:=\{x^{ad}\}\times X$. Let~$(H,X_H)$ be as in Theorem~\ref{thm:weakly inter special}.
Let~$\Phi_1$ and~$\Phi_2$ be the projections of~$(H,X_H)$ on~$(M,X_M)$ and~$(G,X)$. Then~$V=\{x^{ad}\}\times V'$ is a component of~$\stackrel{-1}{\Phi_1}(\{x^{ad}\})=W\cap X_H$, and
\[
V'=\Phi_2(\{x^{ad}\}\times V')\text{ is a component of }\Phi_2\left(\stackrel{-1}{\Phi_1}(\{x^{ad}\})\right).
\]
We now prove the other implication.
Assume that there are morphisms of Shimura data~$\Phi_1:(H,X_H)\to (M^{ad},X^{ad}_M)$ and~$\Phi_2:(H,X_H)\to (G,X)$ such that~$V'$ is a component  of~$\Phi_2\left(\stackrel{-1}{\Phi_1}(\{x^{ad}\})\right)$. The graph of~$\Phi_1$ gives a Shimura subdatum~$(H_1,X_1)\leq (M\times H,X_M\times X_H)$. Then any component~$V_1$ of~$X_1\cap \left(\{x\}\times X_H\right)$ is of the form~$\{x\}\times V'_1$ for a component~$V'_1$ of~$\stackrel{-1}{\Phi_1}(\{x^{ad}\})$.  Theorem~\ref{thm:weakly inter special} implies that~$V'_1\in WS_{X_H}(M,X_M,x)$. Thus~$V'_1\cap \Sigma_{X_H}(M,X_M,x)$ contains a point~$x_1$. We deduce~$\Phi_2(x_1)\in \Phi_2(V'_1)$. But~$\Phi_2(x_1)\in \Sigma_{X}(H,X_H,x_1)$. By remark~\ref{rem:hybrid3} following Definition~\ref{def:orbite hybride}, we have~$\Phi_2(x_1)\in \Sigma_{X}(M,X_M,x)$. Note that~$\Phi_2(V'_1)$ is weakly special. Thus~$V'=\Phi_2(V'_1)$ belongs to~$WS_{X}(M,X_M,x)$.
\end{proof}
We now deduce the functoriality properties.
\begin{corollary}\label{Functoriality}
Let~$\Phi:(G,X)\to (G',X')$ be a morphism of Shimura data.
\begin{enumerate}
\item
For every subset~$Z\subset X$, we have
\[Z\in WS_{X}(M,X_M,x)\Rightarrow\Phi(Z)\in WS_{X'}(M,X_M,x).\]
\item
For every subset~$Z'\subset X'$, we have
\[Z'\in WS_{X'}(M,X_M,x)\Leftrightarrow \forall Z''\in Irr\left(\stackrel{-1}{\Phi}(Z')\right), Z''\in WS_{X}(M,X_M,x),\]
where~$Irr(\stackrel{-1}{\Phi}(Z'))$ denotes the set of complex analytic irreducible components.

In particular, if~$Z_{x'}$ denotes the fibre~$\stackrel{-1}{\Phi}(\{x'\})$,
\[
\forall x'\in \Sigma_{X'}(M,X_M,x),\text{ the components of $Z_{x'}$ belong to~$\in WS_{X}(M,X_M,x)$}.
\]
\item
Let~$\ol{Y}$ denote the closure for the Archimedean topology (resp. the Zariski topology). Then for every subset~$Z'\subseteq \Phi(X)$, we have
\[
\ol{Z'\cap \Sigma_{X'}(M,X_M,x)}=\ol{Z'}\Leftrightarrow\ol{\stackrel{-1}{\Phi}(Z')\cap\Sigma_{X}(M,X_M,x)}=\ol{\stackrel{-1}{\Phi}(Z')}.
\]
In particular,
\[
\forall
x'\in \Sigma_{X'}(M,X_M,x),\ol{Z_{x'}\cap\Sigma_{X}(M,X_M,x)}=\ol{Z_{x'}}.
\]
\end{enumerate}
\end{corollary}

We can now prove the following result.

\begin{theorem}\label{thm: intersection atypiques are the same}
Let~$W\subseteq \Gamma \sous X^+$ be a Hodge generic weakly special subvariety and let~$V\subseteq W$
be a subvariety and let~$Z\subseteq \Gamma\sous X^+$ be a special subvariety.
\begin{enumerate}
\item \label{thm: intersection atypiques are the same1}
 Let~$C$ be an irreducible component of~$Z\cap V$ such that~$\codim_Z(C)<\codim_{\Gamma\sous X^+}(V)$.
Then there exists an irreducible component~$Z'$ of~$Z\cap W$ such that~$\codim_{Z'}(C)<\codim_{W}(V)$ and such that~$C$ is an irreducible component of~$Z'\cap V$.
\item \label{thm: intersection atypiques are the same2}
Let~$Z'$ be an irreducible component of~$Z\cap W$ and let~$C'$ be an irreducible component of~$Z'\cap V$ such that~$\codim_{Z'}(C')<\codim_{W}(V)$.
Then there exists an irreducible component~$C$ of~$Z\cap V$ such that~$\codim_Z(C)<\codim_{\Gamma\sous X^+}(V)$
and~$C'\subseteq C$.
\end{enumerate}
\end{theorem}
\begin{proof}According to Corollary~\ref{cor:special intersections are typical}, for every irreducible component~$Z'$ of~$Z\cap W$, we have
\[
\dim(Z')
=
\dim(Z)+\dim(W)-\dim(\Gamma\sous X^+).
\]
We deduce
\begin{equation}\label{compare exp dim}
\dim(Z')+\dim(V)-\dim(W)=\dim(Z)+\dim(V)-\dim(\Gamma\sous X^+).
\end{equation}

Let us prove assertion~\ref{thm: intersection atypiques are the same1}.

Let~$C$ be an irreducible component of~$Z\cap V$.
We have~$C\subseteq Z\cap V\subseteq Z\cap W$. Thus there exists an irreducible component~$Z'$ of~$Z\cap W$ such that~$C\subseteq Z'$. As~$C$ is a maximal irreducible closed subset of~$Z\cap V$, it is a maximal irreducible closed subset of~$Z'\cap V\subseteq Z\cap V$. Thus~$C$ is an irreducible component of~$Z'\cap V$. 

By~\eqref{compare exp dim}, we have
\begin{align*}
&\dim(C)>
\dim(Z')+\dim(V)-\dim(W)\\
\Leftrightarrow&
\dim(C)>
\dim(Z)+\dim(V)-\dim(\Gamma\sous X^+).
\end{align*}
The assertion~\eqref{thm: intersection atypiques are the same1} follows.

Let us prove~\eqref{thm: intersection atypiques are the same2}.

Let~$Z'$ be an irreducible component of~$Z\cap W$ and let~$C'$ be an irreducible component of~$Z'\cap V$. Then~$C'\subseteq Z'\cap V\subseteq Z\cap V$. Thus there exists an irreducible component~$C$ of~$Z\cap V$ which contains~$C'$. Thus~$\dim(C)\geq \dim(C')$ and, by~\eqref{compare exp dim}, we have
\begin{align*}
&\dim(C')>
\dim(Z')+\dim(V)-\dim(W)\\
\Leftrightarrow&
\dim(C')>
\dim(Z)+\dim(V)-\dim(\Gamma\sous X^+)\\
\Rightarrow&
\dim(C)>
\dim(Z)+\dim(V)-\dim(\Gamma\sous X^+).
\end{align*}
The assertion~\eqref{thm: intersection atypiques are the same2} follows.
\end{proof}

\begin{corollary}\label{coro:equivalence-ZP}
Let~$(M,X_M)$ and~$(G,X)$ be two Shimura data, let~$S=\Gamma_M\sous X_M^+\times \Gamma\sous X^+$ for some arithmetic lattices~$\Gamma_M \leq M(\Q)$ and~$\Gamma\leq G(\Q)$,  let~$x\in X_M^+$ be Hodge generic and let~$W=\Gamma_M\cdot x\times \Gamma\sous X^+\subseteq S$. We denote by~$Irr(~)$ for the set of irreducible components.

Let~$\Sigma_{S}$ be the set of special subvarieties~$Z\subseteq S$.

Then, for~$V=\{\Gamma_M\cdot x\}\times V'\subseteq W$, the sets
\[
ZP_{V,S}:=\bigcup\{C\in Irr(Z\cap V)|Z\in \Sigma_{S}\text{ and } \codim_{Z}(C)<\codim_{S}(V)\}
\]
and
\begin{multline*}
ZP_{V',W}:=\bigcup\{C'\in Irr(Z'\cap V')|Z'\in WS_{\Gamma\sous X^+}(M,X_M,x)\\ \text{ and } \codim_{Z'}(C')<\codim_{\Gamma\sous X^+}(V')\}
\end{multline*}
satisfy
\[
ZP_{V,S}=\{\Gamma_M\cdot x\}\times ZP_{V',W}.
\]
\end{corollary}
\begin{proof} Let us write
\begin{align*}
E&:= \{C\in Irr(Z\cap V)|Z\in \Sigma_{S}, \codim_{Z}(C)<\codim_{S}(V)\}
\\
E'&:=\{C'\in Irr(Z'\cap V')|Z'\in WS_{\Gamma\sous X^+}(M,X_M,x), \codim_{Z'}(C')<\codim_{\Gamma\sous X^+}(V')\}.
\end{align*}

It is enough to prove that
\begin{equation}\label{E et E'}
\forall C\in E, \exists C''\in E', C=\{\Gamma_M\cdot x\}\times C''
\end{equation}
and that 
\begin{equation}\label{E' et E}
\forall C''\in E', \exists C\in E, \{\Gamma_M\cdot x\}\times C''\subseteq C.
\end{equation}

We now prove~\eqref{E et E'}.
Let~$C\in E$ and let~$Z\in \Sigma_S$ such that~$C\in Irr(Z\cap V)$ and~$\codim_{Z}(C)<\codim_{S}(V)$.
Let~$Z'$ be as in~\eqref{thm: intersection atypiques are the same1} of Theorem~\ref{thm: intersection atypiques are the same}. We have
\[
C\in Irr(Z'\cap V)\text{ and }\codim_{Z'}(C)<\codim_{W}(V).
\]
By Theorem~\ref{thm:weakly inter special}, there exists~$Z''\in  WS_{\Gamma\sous X^+}(M,X_M,x)$ such that~$Z'=\{\Gamma_M\cdot x\}\times Z''$.
We deduce~$C'=\{\Gamma_M\cdot x\}\times C''$ with~$C''\in E'$.

We prove~\eqref{E' et E}.
Let~$C''\in E'$ and let~$Z''\in WS_{\Gamma\sous X^+}(M,X_M,x)$ such that~$C'\in Irr(Z'\cap V')$ and~$\codim_{Z'}(C')<\codim_{\Gamma\sous X^+}(V')$.
According to Theorem~\ref{thm:weakly inter special}, there exists~$Z\in \Sigma_S$ such that~$Z':=\{\Gamma_M\cdot x\}\times Z''$ is an irreducible component of~$Z\cap W$. Let~$C$ be as in~\eqref{thm: intersection atypiques are the same2} of Theorem~\ref{thm: intersection atypiques are the same} for~$C':=\{\Gamma_M\cdot x\}\times C''$. Then~$C\in E$ and~$C'\subseteq C$.
\end{proof}

It follows from~Corollary~\ref{coro:equivalence-ZP} that Conjecture~\ref{conj:roro-ZP} (Zilber-Pink) for~$V=\{\Gamma_M\cdot x\}\times V'\subseteq \Gamma_M\sous X_M\times  \Gamma\sous X^+$ is equivalent to Conjecture~\ref{conj:truc} below for~$V'\subseteq \Gamma\sous X^+$.  In particular Conjecture~\ref{conj:truc} for~$(M,X_M,x)=(\{1\},\{1\},1)$ is equivalent to Conjecture~\ref{conj:roro-ZP} for~$V=V'\subseteq \Gamma\sous X^+$.
\begin{conjecture}[Variant of Zilber-Pink for $WS_{\Gamma\sous X^+}(M,X_M,x)$]\label{conj:truc}
Assume~$V'\subsetneq \Gamma\sous X^+$ is an irreducible subvariety. 

Then there exists~$Z_1,\ldots,Z_k\in WS_{\Gamma\sous X^+}(M,X_M,x)$ such that
\[
ZP_{V',W}\subseteq Z_1\cup\ldots \cup Z_k\neq  \Gamma\sous X^+.
\]

\end{conjecture}

The following is straightforward.

\begin{lemma}\label{lem:7.5}
Let~$V'\subseteq \Gamma\sous X^+$ be an irreducible subvariety and write
\[
\Phi:=\{Z'\in WS_{\Gamma\sous X^+}(M,X_M,x)|Z'\subseteq V'\}
\]
and
\begin{multline*}
\Psi:=\{C'\in Irr(Z'\cap V')|Z'\in WS_{\Gamma\sous X^+}(M,X_M,x)\\
\text{ and } \codim_{Z'}(C')<\codim_{\Gamma\sous X^+}(V')\}.
\end{multline*}
Then
\begin{align}
\dim(\Gamma\sous X^+)-\dim(V')>0&\Rightarrow \Phi\subseteq \Psi
\\\text{ and~}\dim(\Gamma\sous X^+)-\dim(V')=1&\Rightarrow\Phi=\Psi.
\end{align}
\end{lemma}

We deduce that
\begin{enumerate}
\item In general, Conjecture~\ref{conj:AO+APZ-bis} (Hybrid conjecture) for~$V'$ is a consequence of Conjecture~\ref{conj:truc} for~$V'$.
\item If~$\dim(\Gamma\sous X^+)-\dim(V')=1$, then Conjecture~\ref{conj:AO+APZ-bis} for~$V'$ is equivalent to Conjecture~\ref{conj:truc} for~$V'$.
\item If~$\dim(\Gamma\sous X^+)-\dim(V')=1$, then Conjecture~\ref{conj:AO+APZ-bis} for~$V'$ is equivalent to Conjecture~\ref{conj:roro-ZP} for~$V=\{\Gamma_M\cdot x\}\times V'\subseteq \Gamma_M\sous X_M^+\times  \Gamma\sous X^+$.
\end{enumerate}
In the same way that the André-Oort Conjecture implies Conjecture~\ref{conj:roro-ZP} for an hypersurface of a special variety,
the Hybrid Conjecture  implies Conjecture~\ref{conj:roro-ZP} for an hypersurface of a weakly special variety.

%

\section{Relation to the work of Aslanyan and Daw} \label{sec:AD}
In~\cite{AD}, V. Aslanyan and  C. Daw study a combination of the André-Oort and the André-Pink-Zannier conjecture \cite[Conj. 3.4--3.5 and Th.~3.9]{AD}. They introduce and study a family of~``$\Sigma$-special subvarieties'' in their terminology (\cite[Def. 3.1--3.2]{AD}).

Let~$\Gamma\sous X^+$ be an arithmetic variety and let~$\Sigma'\subseteq X$ be the subset of all special points.
Let~$\Sigma_{AD}\subseteq \Gamma\sous X^+$ be as in~\cite[Def.~3.1]{AD}. Namely, there exists a finite set~$\{y_1;\ldots;y_k\}\subseteq X$
 and a (possibly infinite) subset~$T\subseteq \Sigma'$ such that
\[
\Sigma=\left(\bigcup_{\tau\in T}\Gamma\sous \cH(\tau)\right)\cup\Gamma\sous \cH(y_1)\cup \ldots\cup\Gamma\sous\cH(y_k).
\]

The following statement implies that~\cite[Conj. 3.4--3.5]{AD} is contained in Conjecture~\ref{conj:AO+APZ}. Note that 
Conjecture~\ref{conj:AO+APZ} contains both André-Pink-Zannier and André-Oort conjecture, which together imply~\cite[Conj. 3.4--3.5]{AD} (see~\cite[Th.~3.9]{AD}).
\begin{lemma}
Let~$(M,X_M)\subseteq (G^k,X^k)$ be the smallest Shimura datum such that~$(y_1,\ldots,y_k)\in X_M$. Then we have
\[
\Sigma'\cup\cH(y_1)\cup \ldots\cup\cH(y_k)
\subseteq \Sigma_X(M,X_M,(y_1,\ldots,y_k))
\]
\end{lemma}
\begin{proof}The remarks~\ref{rem:hybrid0} and~\ref{rem:hybrid4} following Definition~\ref{definition sigma hybride} imply that the set $\Sigma_X(M,X_M,(y_1,\ldots,y_k))$ contains~$\Sigma'$ and is closed under taking Hecke orbits.
It is enough to prove that every~$i\in\{1;\ldots;k\}$, we have~$y_i\in \Sigma_X(M,X_M,(y_1,\ldots,y_k))$.  We may assume~$i=1$. The first projection map~$p_1:(G^k,X^k)\to (G,X)$ is  a morphism of Shimura data which maps~$(y_1,\ldots,y_k)$ to~$y_1$. Let~$(M_1,X_1)\subseteq (G,X)$ the smallest Shimura subdatum such that~$y_1\in X_1$.
As~$(y_1,\ldots,y_k)$ is Hodge generic in~$X_M$, we have~$p_1(X_M)=X_1$. It follows that we have a surjective morphism of Shimura data~$(M^{ad},X^{ad}_M)\to (M^{ad}_1,X^{ad}_1)$ which maps~$(y_1,\ldots,y_k)^{ad}$ to~$y^{ad}_1
$. By~\eqref{definition sigma hybride}, we have~$y_1\in \Sigma_X(M,X_M,(y_1,\ldots,y_k))$. This concludes.
\end{proof}

The following example shows that~\cite[Conj. 3.4--3.5]{AD} does not directly imply Conjecture~\ref{conj:AO+APZ}.

Let us consider a product of arithmetic varieties~$\Gamma\sous X^+=\Gamma_1\sous X_1^+\times\Gamma_2\sous X_2^+$.
Assume that~$\dim(\Gamma_2\sous X_2^+)>1$. Then it is known that there infinitely many special~$\tau_1,\ldots\in X_2^+$
with pairwise non-isomorphic Mumford-Tate groups. It follows that~$\cH(\tau_1),\ldots$ contains infinitely many pairwise disjoint generalised Hecke orbits. Let~$y$ be a non special point of~$X_1$. Then~$\cH((y,\tau_1)),\ldots$ contains infinitely many pairwise disjoints orbits of non-special points. In particular, there does not exists a subset~$\Sigma_{AD}\subseteq \Gamma\sous X^+$ satisfying~\cite[Def.~3.1]{AD} and such that
\[
\Gamma\sous \Gamma\{(y,\tau_1);\ldots\}\in \Sigma_{AD}.
\] 
Because~$(y,\tau)^{ad}$ does not depend on~$\tau$, we have, on the other hand,
\[
\{(y,\tau_1);\ldots\}\subseteq \Sigma_X(G_1,X_1,y).
\]

More generally our sets~$\Sigma_X(M,X_M,y)$ have the following functorial property.

\begin{proposition}Let~$(M,X_M)$, $(G_1,X_1)$ and~$(G_2,X_2)$ be Shimura data, and let~$x\in X_M$ be a Hodge generic point. Then
\[
\Sigma_{X_1\times X_2}(M,X_M,x)=\Sigma_{X_1}(M,X_M,x)\times \Sigma_{X_2}(M,X_M,x).
\]
\end{proposition}
\section{Analogy with the Mordell-Lang conjecture}\label{sec:ML}
Let~$A$ be an abelian variety and~$x\in A$. Given an abelian variety~$B$, let
\[
\mathcal{M}_B(A,x):=\{\phi(x)|\phi\in\Hom(A;B)\}\subseteq B.
\]
When~$B=A$, then~$\mathcal{M}(x)=\End(A)\cdot x$. This is always a subgroup of finite type. If~$\End(A)=\Z$, then~$\mathcal{M}(A,x)$
is the subgroup generated by~$x$ in~$A=B$. 

Assume that~$A=B^k$ with~$k\in\Z_{\geq0}$, and write~$x=(x_1,\ldots,x_k)$. Then~$\mathcal{M}_B(A,x)$ is the~$\End(A)$ submodule  generated by~$\{x_1;\ldots;x_k\}$. If~$\End(A)=\Z$, then~$\mathcal{M}_B(A,x)$ is exactly the subgroup (of finite type) generated by~$\{x_1;\ldots;x_k\}$.

Let~$I$ be the set of~$(A',x')$ where there exists an isogeny~$\iota:A'\to A$ such that~$\iota(x')=x$.
\[
\mathcal{ML}_B(A,x):=\bigcup_{(A',x')\in I}\mathcal{M}_B(A',x')=
\{\phi(x')|(A',x')\in I,\phi\in\Hom(A';B)\}.
\]

We have the following analogy. The couple~$(A,x)$ corresponds to a pointed Shimura datum~$(M,X_M,x)$.
And a couple~$(A',x')\in I$ corresponds to a Shimura datum~$(M',X_{M'},x')$ such that there exists an isomorphism~$(M'^{ad},X^{ad}_{M'})\to (M^{ad},X^{ad}_{M})$ mapping~$x'^{ad}$ to~$x^{ad}$. In this analogy, the  set~$\mathcal{M}_B(A,x)$
corresponds to a generalised Hecke orbit~$\cH_{\Gamma\sous X^+}(M,X_M,x)$ and the set~$\mathcal{ML}_B(A,x)$ corresponds to the hybrid Hecke orbit~$\Sigma_{\Gamma\sous X^+}(M,X_M,x)$. The set of torsion points of~$B$ is~$\mathcal{ML}_B(A,0)$
and the set special points in~$\Gamma\sous X^+$ is~$\Sigma_{\Gamma\sous X^+}(\{1\},\{1\},1)$, where~$(\{1\},\{1\},1)$ is the trivial Shimura datum.

In this analogy, the the André-Oort conjecture is an analogue of Manin-Mumford conjecture\footnote{Proved by Raynaud.}, the André-Pink-Zannier conjecture is an analogue of the so-called Mordellic part\footnote{In the words of~\cite{McQ}. This Mordellic part is proved by Faltings~\cite{FaFaFa}} of the Mordell-Lang conjecture. The analogue of the Mordell-Lang conjecture\footnote{Proved by McQuillan~\cite{McQ}.} is then Conjecture~\ref{conj:AO+APZ}. 
Mordell-Lang conjecture contains both Manin-Mumford conjecture and its Mordellic part, but is not a direct consequence of a conjunction of Manin-Mumford conjecture and its Mordellic part. Similarly Conjecture~\ref{conj:AO+APZ} contains the André-Oort conjecture and the André-Pink-Zannier conjecture, but is not a direct consequence of the conjunction of the André-Oort conjecture and the André-Pink-Zannier conjecture. 
This can be compared to ~\cite[Conj. 3.4--3.5]{AD} (cf~\cite[Th.~3.9]{AD}).


\begin{thebibliography}{99}

%

\bibitem{AD} V. Aslanyan, C. Daw {\it A note on unlikely intersections in Shimura varieties}
\url{https://arxiv.org/abs/2209.07967}

\bibitem{BailyBorel}  W. L. Baily Jr., A.~Borel, {\it Compactification of arithmetic quotients of bounded symmetric domains}, Ann. Math. 84 (1966), pp.~442--528.

%
%
%
%
%
%
%
%
%
%
%
%
%
%
%
%
%
%
%


\bibitem{Deligne} P.~Deligne, {\it Vari\'{e}t\'{e}s de Shimura: interpr\'{e}tation modulaire, et techniques de construction de mod\`{e}les canoniques.} in {\it Automorphic forms, representations and L-functions,} Proc. Sympos. Pure Math., XXXIII, Part 2, pp. 247--289. 


\bibitem{FaFaFa} G. Faltings, {\it 
Diophantine approximation on abelian varieties}
Pages 549-576 from Volume 133 (1991), Issue 3.


%
%
%
%
%
%
%
%
%

\bibitem{McQ}  M. McQuillan, {\it M. Division points on semi-abelian varieties.} Invent Math 120, 143–159 (1995). \url{https://doi.org/10.1007/BF01241125}

\bibitem{Milne} J. Milne {\it Lectures on Shimura varieties.} Lecture notes available on author's web-page.

\bibitem{Moonen} B.~Moonen, {\it Linearity properties of Shimura varieties. I.},
J. Algebraic Geom.~7 (1998), no. 3, pp.~539--567.


\bibitem{OrrThesis} M. Orr, {\it La conjecture de Andr\'e-Pink: orbites de Hecke et sous-vari\'et\'es faiblement sp\'eciales.} 
Phd thesis, University Paris-Sud (Orsay), 2013.

\bibitem{AO-PST}
J. Pila, A. N. Shankar, J. Tsimerman 
with an appendix by H. Esnault and M. Groechenig,
{\it Canonical Heights on Shimura Varieties and the André-Oort conjecture}. \url{https://arxiv.org/abs/2109.08788v3}

\bibitem{P} J. Pila, {\it Point-counting and the Zilber-Pink conjecture.}
Cambridge Tracts in Math., 228, Cambridge University Press, Cambridge, 2022.

%
%
%
%
%
%
%
%

\bibitem{APZ2} R. Richard, A. Yafaev,  
{\it Generalised André-Pink-Zannier conjecture for Shimura varieties of Abelian Type.}  \url{https://arxiv.org/abs/2111.11216}

\bibitem{APZ2-CRAS} R. Richard, A. Yafaev,  
{\it Generalised André-Pink-Zannier conjecture for Shimura varieties of Abelian Type.}  Submitted to Comptes Rendus Math. Acad. Sci. Paris.

\bibitem{APZ1} R. Richard, A. Yafaev,  
{\it Height functions on Hecke orbits and the generalised André-Pink-Zannier conjecture.}  \url{https://arxiv.org/abs/2109.13718}


\bibitem{APZ1-CRAS} R. Richard, A. Yafaev,  
{\it On the Generalised André-Pink-Zannier conjecture.} To appear in Comptes Rendus Math. Acad. Sci. Paris.


\bibitem{RY-CRAS} R. Richard, A. Yafaev,  
%
%
%
%
%
%
%
%
%
%
%
%

\bibitem{U} E. Ullmo, {\it Equidistribution des sous-vari\'et\'es sp\'eciales II. }, 
J. Reine Angew. Math. 606 (2007), p. 193-216.

\bibitem{U1} E. Ullmo, {\it Quelques applications du th\'eor\`eme d'Ax-Lindemann hyp\'erbolique.}, Compositio Mathematica. 150 (2014), no. 2, 175–190.

%
%
%
%
%

\end{thebibliography}
\end{document}